\documentclass[10.5pt,reqno]{amsart}
\usepackage{amsmath,amssymb,doublespace}
\usepackage{epsfig,enumerate,multirow}
\usepackage{array}

\newtheorem{thm}{Theorem}[section]

\newtheoremstyle{theorem}%name
  {10pt}          % space above
  {10pt}  % space below theoremnhb
  {\normalfont}  % body font
  {}     % ident - empty=no indent,  \parindent= paragraph indent
  {\bf}  % thm head font
  {. }    % punctuation after thm head
  { }    % space after thm head: `` ``=normal \newline=linebreak
  {}     % thm head specification

\theoremstyle{theorem}
\newtheorem{dfn}{Definition}[section]
\newtheorem{rem}{Remark}[section]

\makeatletter
\def\section{\@startsection{section}{1}%
  \z@{.7\linespacing\@plus\linespacing}{.5\linespacing}%
  {\normalfont\bfseries}}

\def\@seccntformat#1{%
  \protect\textbf{\protect\@secnumfont
    \csname the#1\endcsname
    \protect\@secnumpunct
  }}

\def\@setauthors{%
  \begingroup
  \def\thanks{\protect\thanks@warning}%
  \trivlist
  \centering \small  \@topsep30\p@\relax
  \advance\@topsep by -\baselineskip
  \item\relax
  \author@andify\authors
  \def\\{\protect\linebreak}%
  \authors%\\\addresses%
  \ifx\@empty\contribs
  \else
    ,\penalty-3 \space \@setcontribs
    \@closetoccontribs
  \fi
  \endtrivlist
  \endgroup
}

\def\@settitle{\begin{center}%
  \baselineskip14\p@\relax
    \bfseries\large
  \@title
  \end{center}%
}

\def\@adjustvertspacing{%
  \bigskipamount.4\baselineskip plus.4\baselineskip
  \medskipamount\bigskipamount \divide\medskipamount\tw@
  \smallskipamount\medskipamount \divide\smallskipamount\tw@
  \abovedisplayskip\bigskipamount
  \belowdisplayskip \abovedisplayskip
  \abovedisplayshortskip\abovedisplayskip
  \advance\abovedisplayshortskip-1\abovedisplayskip
  \belowdisplayshortskip\abovedisplayshortskip
  \advance\belowdisplayshortskip 1\smallskipamount
  \jot\baselineskip \divide\jot 4 \relax
}

\makeatother

\numberwithin{equation}{section}

\begin{document}

\textit{Journal of the Kerala Statistical Association}, Vol. 10,  December 1999, p. 01-07.\\

\begin{spacing}{1.35}
\title{On Geometric Infinite Divisibility}

\author[E S\lowercase{andhya and} R N P\lowercase{illai}]
{\textbf{E Sandhya}\\
Department of Statistics, Prajyoti Niketan College,
Pudukkad, Thrissur - 680 301, India.\\
\mbox{}\\
\lowercase{and}\\
\mbox{}\\
\textbf{R N Pillai}\\
Valiavilakam, Ookode P.O., Vellayani, Trivandrum, India.}

\begin{abstract}
The notion of geometric version of an infinitely divisible law is introduced. Concepts parallel to attraction and partial attraction 
are developed and studied in the setup of geometric summing of random  variables.
\end{abstract}

\maketitle

\section{Inroduction}
Klebanov, \textit{et al.} \cite{kl-et-84} deﬁned;

\begin{dfn}\label{dfn-1.1} % Deﬁniti0n.1.1 
A random variable (r.v) $X$ is geometrically infinitely divisible
(GID) if for every $p\in (0,1)$, $X\overset{d}{=} X^{(p)}_1+
\ldots + X^{(p)}_{N_p}$, where $N_p$ and $\{X^{(p)}_i\}$ are
independent, $\{X^{(p)}_i\}$ are i.i.d and $N_p$ is a geometric 
r.v  with mean $1/p$.
\end{dfn}

Equivalently, a r.v $X$ with characteristic function (c.f) $\phi(t)$ is GID if and only if $\exp\{1-1/\phi(t)\}$ is an infinitely divisible (ID) c.f. They also introduced geometrically strictly stable (GSS) laws as:

\begin{dfn}\label{dfn-1.2} %Deﬁnition.1.2 
A r.v $Y$ is GSS if for every $p\in (0,1)$, there exists a constant
$c(p) > 0$ such that $Y\overset{d}{=} c(p)\{X_1+\ldots+X_{N_p}\}$, where $N_p$ and $\{X_i\}$ are independent, $\{X_i\}$ are i.i.d, 
$Y\overset{d}{=} X_1$ and $N_p$ is a geometric r.v with mean $1/p$.
\end{dfn}

Pillai \cite{pi-85} introduced semi-$\alpha$-Laplace laws as:

\begin{dfn} %Definiti0n.1.3 
A distribution function (d.f) $F$ with c.f 
$\phi(t)=1/(1+g(t))$ is semi-$\alpha$-Laplace of exponent 
$\alpha$, $0 < \alpha\leq 2$, if $g(t)=ag(bt)$ for some constants
$a$ and $b$, $0 < b < 1 < a$ and $\alpha$ is the unique solution of $ab^\alpha=1$, $b$ is called the order of the 
semi-$\alpha$-Laplace law.
\end{dfn}

\noindent\textbf{Note.} 
If $b_1$ and $b_2$ are orders of $F$ such that $\log b_1/\log b_2$ is irrational, then $g(t)=c|t|^\alpha$ for some constant $c>0$.

\begin{rem} % Remark.1.1 
The result and discussion in Kagan \textit{et al.} 
\cite[p.323,324]{ka-et-73} is relevant in this context.
\end{rem}

This study is motivated by the one-to-one correspondence between ID and
GID laws (see the equvalent of definition \ref{dfn-1.1}) and extend 
concepts of attraction and partial attraction to the class of GID laws.

It is known that the class of GID laws is a proper subclass of the class of
ID laws, Sandhya \cite{sa-91a}. The idea of geometric compounding is related to $p$-thinning of point processes which has applications in the study of patterns of occurrences of crimes and accidents as many of them go unreported. When we consider $p$-thinning of renewal processes for 
every $p\in (0,1)$, we get Cox and renewal processes, which are used in modelling occurrences of claims in insurence business, Grandall 
\cite{gr-91}. Sandhya \cite{sa-91b} identiﬁed the class of renewal 
processes, invariant under $p$-thinning, as those with semi-Mittag-Leffler laws as the interval distribution and observed that, a renewal process is Cox if and only if its interval distribution is GID. Aspects of geometric compounding and GID laws in system reliability studies are explored in Pillai and Sandhya \cite{pisa-96}. Since exponential mixtures are GID (Pillai and Sandhya \cite{pisa-90}), construction of distributions
useful in situations mentioned above is easier.

In section 2 the ideas of geometric version of an ID law and geometric attraction are introduced and studied. The notion of partial geometric attraction is developed in section 3.

A detailed presentation of these ideas and a generalization of GID laws to
$\nu$-ID laws are available in Sandhya \cite{sa-91a}. For a different approach to geometric attraction see Mohan \textit{et al.} 
\cite{mo-et-93}. Klebanov and Rachev \cite{klra-96} develop 
$\nu$-infinite divisibility, define geometric attraction for random vectors and the domains of attraction of $\nu$-stable random vectors. 
Kozubowski and Rachev \cite{kora-99} discuss multivariate geometric stable laws. Semigroup related to random stability was studied by 
Bunge \cite{bu-96}.

\section{Geometric Version and Geometric Attraction} 
\begin{dfn}\label{dfn-2.1} % Deﬁnition.2.1 
Let $F$ be an ID d.f with c.f $\exp\{-g(t)\}$. Then the d.f $G$ is the
geometric version (g.v) of $F$ if and only if $G$ has c.f 
$1/(1 + g(t))$.
\end{dfn}

\begin{rem} % Remark.2.1 
Definition \ref{dfn-2.1} leads to a one-to-one correspondence between the
c.f’s of an ID law and its g.v. It may also be noticed that the relation between an ID law and its g.v basically is the relation between a compound Poisson law and the corresponding compound geometric law.
\end{rem}

\begin{rem} % Remark.2.2 
The g.v of an ID law is GID and hence ID.
\end{rem}

\begin{thm} %Theorem.2.1 
A d.f $G$ is the g.v of an ID law $F$ with c.f $\exp\{-g(t)\}$, if and
only if $G$ is the limit law of geometric sums of the form 
$U_n=X_{1,n}+\ldots+X_{N_n,n}$ where $N_n$ a geometric r.v with mean $n$, is independent of $\{X_{i,n}\}$ and $\{X_{i,n}\}$ is i.i.d with 
c.f $\exp\{-g(t)/n\}$.
\end{thm}

\begin{proof}
If $\phi_{U_n}(t)$ denotes the c.f of $U_n$, then we have,
$$
\phi_{U_n}(t) \rightarrow 1/(1+ g(t)) \quad \text{as}\quad 
n\rightarrow \infty.
$$
The converse is obtained by retracing the steps.
\end{proof}

\begin{dfn} % Deﬁnition.2.2 
A d.f $F$ is said to be geometrically attracted to another d.f $G$, if
$G$ is the limit law of a sum 
%%%%%
\begin{equation}\label{eq-2.1}
Y_n= B_n^{-1}\{X_1+\ldots+ X_{Np_n}\}
\end{equation}
%%%%%
where $Np_n$, which is geometric with mean $1/p_n$, is independent of 
$\{X_i\}$ which are i.i.d as $F,B_n > 0$ and as $n\rightarrow\infty$,
$p_n\rightarrow 0$ and $B_n\rightarrow\infty$.
\end{dfn}

\begin{dfn} % Deﬁnition.2.3 
The set of all d.f’s that are geometrically attracted to $G$ is called
the domain of geometric attraction (d.g.a) of $G$.
\end{dfn}

\begin{thm} % The0rem.2.2 
Every GSS law is geometrically attracted to itself. Conversely,
a d.f $F$ belongs to the d.g.a of a GSS law provided the d.f of 
$Y_n$ in \eqref{eq-2.1} with 
$B_n=p^{-n/\alpha}$, $0 <\alpha \leq 2$, and 
$p_n=p^n$ tends to a limit for two values of $p$ say $p_1$ and 
$p_2$, such that $\log p_1/\log p_2$ is irrational.
\end{thm}

\begin{proof}
Let $F$ be GSS. By a little algebra it follows from 
definition \ref{dfn-1.2} that its c.f $\phi(t)$ satisfies, 
for every $p\in (0,1)$,
$$
\phi(t) = \frac{p\phi(ct)}{1-q\phi(ct)},\quad
q=1-p,
$$
with $c=p^{1/\alpha}$, $0 < \alpha\leq 2$. And on iteration,
$$
\phi(t) = \frac{p^n\phi(p^{n/\alpha}t)}{1-(1-p^n)\phi(p^{n/\alpha}t)}.
$$
The right hand side corresponds to the c.f of $Y_n$ in 
\eqref{eq-2.1} with $B_n=p^{-n/\alpha}$ and 
$p_n=p^n$, for all $n>1$ and hence $F$ is geometrically attracted 
to itself. Conversely,
if $Y_n=p^{n/\alpha}\{X_1+\ldots+X_{Np^n}\}$, then its c.f is
$$
\phi_{Y_n}(t) = \frac{1}{1+p^{-n} g(p^{n/\alpha}t)}
$$
where $g=(1/\phi)-1$, $\phi$ being the c.f of $X_1$. 
Suppose $\phi_{Y_n}(t)\rightarrow 1/(1+g_1(t))$, then
$$
\phi_{Y_{n+1}}(t)\rightarrow \frac{1}{1+p^{-1} g_1(p^{1/\alpha}t)}=
\frac{1}{1+g_1(t)}.
$$
Hence,

\begin{equation}\label{eq-2.2}
g_1(t) = p^{-1} g_1(p^{1/\alpha}t).
\end{equation}

If for two values of $p$, say $p_1$ and $p_2$, such that 
$\log p_1/\log p_2$ is irrational, \eqref{eq-2.2} is
satisﬁed then by Kagan \textit{et al.} \cite[p.324]{ka-et-73}, 
the converse follows.
\end{proof}

With $[x]$ denoting the integer part of $x$, we have;

\begin{thm} % Theorem.2.3 
A necessary condition that $F$ is in the d.g.a of $G$ which is
GSS, is that $[1/p_n]=n$. Conversely, if $F$ 
is in the d.g.a of another law $G$ and $[1/p_n]=n$, then $G$ is GSS.
\end{thm}

\begin{proof}
Let $F$ belong to the d.g.a of $G$. In terms 
of the corresponding c.f’s we have, 
$$
\frac{1}{1+p^{-1}_ng(t/B_n)} \rightarrow
\frac{1}{1+g_1(t)}.
$$
That is,
$$
\{[1/p_n]+\theta_n\} g(t/B_n) \rightarrow g_1(t)
$$
where $0\leq \theta_n=p^{-1}_n-[p^{-1}_n]\leq 1$. Hence
$$
[1/p_n] g(t/B_n)\rightarrow g_1(t).
$$
This implies that $[1/p_n]=n$, since now
$$
\exp\{-[1/p_n]g(t/B_n)\} \rightarrow \exp \{-g_1(t)\}
$$
which is strictly stable. Converse follows by retracing the steps.
\end{proof}

\section{Partial Geometric Attraction}
\begin{dfn} % Deﬁnition. 3.1 
A d.f $F$ is said to be partially geometrically attracted to another 
d.f $G$ if there exists a subsequence $k_1< k_2<\ldots< k_n<\ldots$ 
of the set of positive integers such that $Y_{k_n}\rightarrow Y$ 
in law, where $Y_n$ is defined in \eqref{eq-2.1} and $G$ is the d.f
of $Y$.
\end{dfn}

\begin{dfn} %Deﬁnition. 3.2 
The set of all d.f’s that are partially geometrically attracted to
a d.f $G$ is called the domain of partial geometric attraction (d.p.g.a) 
of $G$.
\end{dfn}

\begin{thm} %Theorem.3.1 
Corresponding to every semi-$\alpha$-Laplace law there exists an
ID law that is partially geometrically attracted to it. Conversely, 
if an ID law is partially geometrically attracted to some law $G$ 
with $k_{n+1}/k_n\rightarrow a$, and $B_{k_{n+1}}/B_{k_n}\rightarrow 1/b$, then $G$ is semi-$\alpha$-Laplace of order $b$ and $ab^\alpha=1$.
\end{thm}
\begin{proof}
Suppose $Y$ is semi-$\alpha$-Laplace with d.f $G$. Hence its 
c.f $\phi$ satisfies
$$
\phi(t) = \frac{1}{1+g(t)} = 
\frac{1}{1+a^n g(b^n t)}, 0<b<1<a.
$$
Choose a subsequence $\{k_n\}$ such that $k_n=[a^n]$ and 
$B_{k_n}=b^{-n}$ for all $k$ and $n$, positive integers. Let
$$
\phi_{k_n}(t)=\{1+[a_n]g(b^nt)\}^{-1}
$$
and put $\theta_n=a^n-[a^n]$. Now,
$$
|\phi_{k_n}(t)-\phi(t)|=|\theta_n g(b^nt)||\phi_{k_n}(t)|
|\phi(t)|.
$$
Since $\phi_{k_n}(t)$ and $\phi(t)$ are c.f’s
$$
|\phi_{k_n}(t)| |\phi(t)|\leq 1
$$
and $g(b^nt)\rightarrow 0$ as $n\rightarrow\infty$ and hence 
$\phi_{k_n}(t)\rightarrow \phi(t)$. But $\phi_{k_n}(t)$ is the g.v of
$\exp\{-[a^n]g(b^nt)\}$ and it corresponds to the r.v 
$Y_{k_n}=b^n\{X_1+\ldots+X_{N_{1/[a^n]}}\}$ and $Y_{k_n}\rightarrow Y$ in law.

Conversely, suppose that an ID law $F$ is in the d.p.g.a of $G$ 
with $k_{n+1}/k_n\rightarrow a$ and $B_{k_{n+1}}/B_{k_n}\rightarrow 1/b$. Choose $k_n=[a^n]$ and $B_{k_n}=1/b^n$ for all $k$ and $n$ positive
integers. That is,
$$
\{1+[a^n] g(b^nt))\}^{-1}\rightarrow 
\{1+g_1(t)\}^{-1},
$$
where the L.H.S is the c.f of $Y_{k_n}$ and the R.H.S that of $G$. 
Now the c.f of $Y_{k_{n+1}}$ is
$$
\{1+[a^{n+1}] g(b^{n+1}t)\}^{-1}
$$
and
%%%%
\begin{align*}
\lim_{n\rightarrow\infty}
[a^{n+1}]g(b^{n+1}t) & = \lim_{n\rightarrow\infty}
[a^n]g(b^{n+1}t) + \lim_{n\rightarrow\infty} 
\{a^n(a-1)+\theta_n-\theta_{n+1}\}g(b^{n+1}t)\\
& = g_1(bt) + (a-1) g_1(bt)\\
& = ag_1(bt).
\end{align*}
%%%%
That is, $\{1+g_1(t)\}^{-1}=\{1+ag_1(bt)\}^{-1}$, implying $G$ 
is semi-$\alpha$-Laplace, and the proof is complete. 
\end{proof}

\begin{thm} %Theorem.3.2 
An ID law $F$ is in the domain of partial attraction of a semi-stable 
law $G$ having order $b > 0$, if and only if its g.v is in the 
d.p.g.a of a semi-$\alpha$-Laplace law which is the g.v of $G$.
\end{thm}

\begin{proof}
Follows from the definitions of partial attraction, partial geometric
attraction and g.v of an ID law.
\end{proof}

Gnedenko and Kolmogorov \cite{gnko-68} gave a transitive relation for a 
law to belong to the domain of partial attraction of another law. An analogous result for d.p.g.a is given below, the proof of which is straight forward.

\begin{thm} % Theorem.3.3 
If a law $F'$ is in the d.p.g.a of the law $G'$, and $G'$ is in the
d.p.g.a of the law $H'$, then $F'$ is in the d.p.g.a of the law 
$H'$, if $F'$, $G'$ and $H'$ are the g.v’s of some ID laws 
$F,G$ and $H$ respectively.
\end{thm}

\subsection*{Acknowledgements} 
The authors thank the referee and the editor for some useful suggestions. The first author's work was supported by a research fellowship from 
C.S.I.R, India.

\end{spacing}
\end{document}